\documentclass[12pt]{article}
\usepackage[margin=1.5in]{geometry} 
\usepackage{amsmath,amsthm,amssymb,amsfonts}
\usepackage{changepage}
\usepackage{tikz}
\usepackage[shortlabels]{enumitem}
\usetikzlibrary{matrix}

\newcommand{\D}{\mathcal{D}}
\newcommand{\CC}{\mathcal{C}}

\newcommand{\Z}{\mathbb{Z}}
\newcommand{\N}{\mathbb{N}}

\newcommand{\id}{\textnormal{id}}

\newcommand{\inv}{^{-1}}

\newcommand{\cay}{\textnormal{Cay}}

\newtheorem{theorem}{Theorem}

\newtheorem{lemma}{Lemma} 
\newtheorem{prop}{Proposition}  
\newtheorem{prob}{Problem} 
 
\begin{document}
 
\title{Marked Groups with Isomorphic Cayley Graphs but Different Borel Combinatorics}
\author{Felix Weilacher}
\maketitle

\begin{abstract}
\noindent We construct pairs of marked groups with isomorphic Cayley graphs but different Borel chromatic numbers for the free parts of their shift graphs. This answers a question of Kechris and Marks. We also show that these graphs have different Baire measurable and measure chromatic numbers, answering analogous versions of the question.
\end{abstract}


\section{Introduction}\label{sec:intro}


Let $X$ be a set, and $G \subset X \times X$ a symmetric irreflexive relation on $X$. $G$ is called a \textit{graph} on $X$ and the elements of $X$ are called the \textit{vertices} of $G$. When the meaning is clear, we shall sometimes identify a graph with its set of vertices. A pair $\{x,y\}$ of distinct elements of $X$ with $(x,y) \in G$ is called an \textit{edge} of $G$, and when this is the case $x$ and $y$ are said to be $\textit{adjacent}$. The equivalence classes of the equivalence relation on $X$ generated by $G$ are called the \textit{connected components} of $G$. A map $c:X \rightarrow Y$ is called a $\textit{coloring}$ of $G$ if it sends adjacent elements of $X$ to distinct elements of $Y$. The elements of $Y$ are called \textit{colors}. If $|Y| = k$, it is called a \textit{$k$-coloring}. The \textit{chromatic number} of $G$, denoted $\chi(G)$, is defined as the least $k$ such that $G$ admits a $k$-coloring.

A large portion of the emerging field of descriptive graph combinatorics investigates these notions in the Borel setting: Let $X$ now be a standard Borel space. A graph $G$ on $X$ is called a \textit{Borel graph} if $G$ is Borel in the product space $X \times X$. A coloring $c:X \rightarrow Y$ is called a $\textit{Borel coloring}$ if $Y$ is also a standard Borel space and $c$ is a Borel map. The $\textit{Borel chromatic number}$ of $G$, denoted $\chi_B(G)$ is defined as the least $k$ such that $G$ admits a Borel $k$-coloring. For a survey covering these and related ideas, see $\cite{KM}$.

This paper will focus on Borel graphs that arise in the following way: Let $\Gamma$ be a finitely generated group and $S \subset \Gamma$ a finite symmetric set of generators with $\id \notin S$. The pair $(\Gamma,S)$ is called a \textit{marked group}. All groups in this paper are assumed to be finitely generated. When there is no confusion as to what the generators are, we will sometimes identify a marked group with the underlying group. Let $a:\Gamma \times X \rightarrow X$ be a free Borel action of $\Gamma$ on some standard Borel space $X$. Form the Borel graph $G(S,a)$ on $X$ by connecting $x,y \in X$ exactly when $s \cdot x = y$ for some $s \in S$.

For every such marked group $(\Gamma,S)$, one such graph is especially important: Give $2^\Gamma$ the product topology, and let $\Gamma$ act on it via the \textit{left shift action} $a_0$:
\begin{equation}
(g \cdot x)(h) = x(g\inv h) \ \ \ \textnormal{for } g,h \in \Gamma, x \in 2^\Gamma 
\end{equation}
This action is not free unless $\Gamma$ is trivial, but its \textit{free part},
\begin{equation}
F(2^\Gamma) = \{x \in 2^\Gamma \mid g \cdot x \neq x, \ \forall g \in \Gamma \textnormal{ with } g \neq \id \}, 
\end{equation}
is a $G_\delta$ subset of $2^\Gamma$, and so we can consider the restriction of the left shift action to it. By \cite{ST}, for any $a$ and $X$ as in the previous paragraph, there is a Borel $\Gamma$-equivariant map $X \rightarrow F(2^\Gamma)$. Therefore studying the graph $G(a_0,S)$, called the \textit{shift graph} of $\Gamma$, can give us information about all of the graphs $G(a,S)$. In particular, we always have $\chi_B(G(a,S)) \leq \chi_B(G(a_0,S))$. We will always refer to the shift graph of $\Gamma$ by its underlying set $F(2^\Gamma)$.

Let $\cay(\Gamma)$, called the \textit{Cayley graph} of $\Gamma$, be the graph $G(a_1,S)$, where $a_1$ is the action of $\Gamma$ on itself given by left multiplication. Clearly, as a graph, $F(2^\Gamma)$ is a disjoint union of copies of $\cay(\Gamma)$, uncountably many if $\Gamma$ is infinite and finitely many otherwise. Thus, if $\Gamma$ is finite, Borel combinatorial properties of $F(2^\Gamma)$ can be understood fully by studying $\cay(\Gamma)$. In particular
\begin{equation}
\chi_B(F(2^\Gamma)) = \chi(F(2^\Gamma)) = \chi(\cay(\Gamma)).
\end{equation}
This will not hold in general when $\Gamma$ is infinite, but one still might hope to get Borel combinatorial information from the Cayley graph in such cases. For example, in \cite{KM} (Problem 5.36), the following question is raised:

\begin{prob}\label{prob:main}
Let $\Gamma$ and $\Delta$ be two marked groups with isomorphic Cayley graphs. Must we have $\chi_B(F(2^\Gamma)) = \chi_B(F(2^\Delta))$?
\end{prob}

Note that by the previous paragraph $\Gamma$ and $\Delta$ must be infinite if this equality is to fail. If instead of (vertex) colorings we consider edge colorings, that is, maps $c:G \rightarrow Y$ with $c(x,y) = c(y,x)$ for all edges $\{x,y\}$ and $c(x,y) \neq c(y,z)$ unless $x = z$, the analogue of Problem \ref{prob:main} has an easy negative answer. Specifically, if we consider the marked groups given by $\Z$ and $\Z/2\Z * \Z/2\Z$ with their standard generators, the shift graph of the former has no Borel edge 2-coloring but the shift graph of the latter does. See \cite{KST} for a discussion of this. This is unsurprising, though, since when working with shift graphs we are able to label our edges with the generator(s) that they correspond to in a Borel fashion. 

In this paper we prove the perhaps more surprising result that Problem \ref{prob:main} has a negative answer as well. Specifically, we prove the following:

\begin{theorem}\label{th:main}
Let $k \geq 3$. There exist marked groups $\Gamma_k$ and $\Delta_k$ with isomorphic Cayley graphs but for which \begin{enumerate}[label=(\roman*)]
	\item $\chi_B(F(2^{\Gamma_k})) = \chi(F(2^{\Gamma_k})) = \chi(F(2^{\Delta_k})) = \chi(\cay(\Gamma_k)) = \chi(\cay(\Delta_k)) = k,$  
	\item $\chi_B(F(2^{\Delta_k})) = k+1$.
\end{enumerate}
\end{theorem}

Obviously a shift graph will not have chromatic number 1 unless the group is trivial. To rule out an example where one of the Borel chromatic numbers is 2, we have the following proposition pointed out to us by A.S. Kechris:

\begin{prop}\label{prop:2}
Let $(\Gamma,S)$ be an infinite marked group. Then $\chi_B(F(2^\Gamma)) > 2$. 
\end{prop}
\begin{proof}
Suppose, to the contrary, we have a Borel 2-coloring $c:F(2^\Gamma) \rightarrow \{1,2\}$. Since $\cay(\Gamma)$ must also be 2-colorable, let $\varphi:\Gamma \rightarrow \{1,2\}$ be a 2-coloring of it such that $\varphi(\id) = 2$. Then note that $\varphi(sg) = 1+\varphi(g)$ (mod 2) for each $g \in \Gamma$ and $s \in S$, and so $\varphi$ is actually a nontrivial homomorphism into $\Z/2\Z$. (This observation is from \cite{CK}.)

Let $A_i = c\inv(i)$ for $i=1,2$. Note that both $A_i$'s are invariant under the action of $\ker \varphi$. Since $\ker \varphi$ is index 2 in $\Gamma$, it is infinite. This implies it has the following homogeneity property: For any nonempty open sets $U,V \subset F(2^\Gamma)$, there is an element $g \in \ker \varphi$ such that $g \cdot U \cap V \neq \emptyset$. Therefore, by Theorem 8.46 in \cite{K95}, each $A_i$ is meager or comeager. Since they partition $F(2^\Gamma)$, one must be meager and the other comeager. However, since $s \cdot A_1 = A_2$ for any $s \in S$, this is a contradiction.
\end{proof}

We'll prove the $k=3$ case of Theorem \ref{th:main} in Section \ref{sec:ex} and the full theorem in Section \ref{sec:big}. 

In Section \ref{sec:iso} we'll add to the surprise of the theorem by proving that for each pair $\Gamma_k$ and $\Delta_k$ constructed in this paper, there is a Borel bijection $f:F(2^{\Gamma_k}) \rightarrow F(2^{\Delta_k})$ which preserves connected components, that is, such that $f(x)$ and $f(y)$ are in the same connected component if and only if $x$ and $y$ are. Thus, even though the connected components of our two graphs can be matched up in a Borel fashion and are also isomorphic as graphs, they differ in an invariant as simple as the Borel chromatic number.

We end the paper by discussing the analogue of Problem \ref{prob:main} in the measurable setting in Section \ref{sec:measure} and posing some new problems related to strengthening Theorem \ref{th:main} in Section \ref{sec:probs}.


\section{A First Example}\label{sec:ex}


We start with the simplest example known to us of a pair of marked groups $\Delta$ and $\Gamma$ with isomorphic Cayley graphs but different Borel chromatic numbers for their shift graphs. We separate these from the rest so that a reader simply interested in getting an answer to Problem \ref{prob:main} can do so without bothering with the details necessary to get every $k$ in Theorem \ref{th:main}.

We first recall some simple facts and notions which will be useful here and throughout. If $G$ is a graph on $X$ and $x \in X$, then by the \textit{degree} of $x$ we mean the number of elements of $X$ adjacent to $x$. We say $G$ has \textit{bounded degree $d$} if the degree of each element of $X$ is at most $d$. Note that if $(\Gamma,S)$ is a marked group and $|S| = d$, every element of the shift graph has degree exactly $d$. 

Let $x,y \in X$. A \textit{path} from $x$ to $y$ in $G$ is a finite sequence $x_0,x_1,\ldots,x_N$ of points in $x$ such that $x = x_0$,$y = x_N$, and $(x_i,x_{i+1}) \in G$ for each $0 \leq i < N$. We call $N$ the \textit{length} of the path. We define $\rho_G(x,y)$, the \textit{path distance} between $x$ and $y$, to be the shortest length of any path from $x$ to $y$, or $\infty$ if there is no such path. If $A,B \subset X$ are nonempty, we define $\rho_G(A,B)$ as the minimum of $\rho_G(x,y)$ taken over all $x \in A$, $y \in B$. Finally, we call a set $A \subset X$ \textit{$(N,G)$-discrete} if $\rho_G(x,y) > N$ for all distinct $x,y \in A$. When the graph is clear from the context, we will just write `$\rho$' and `$N$-discrete'. A 1-discrete set is also called \textit{independent}. Note that a coloring is exactly a map for which the pre-image of each point is independent.

\begin{lemma}\label{discrete}
Let $G$ be a Borel graph on a standard Borel space $X$ of bounded degree $d < \aleph_0$. Then for any finite $N$, $X$ has a maximal Borel $(N,G)$-discrete subset. 
\end{lemma}
\begin{proof}
Define a graph $G'$ on $X$ by $(x,y) \in G'$ if and only if $x \neq y$ and $\rho_G(x,y) \leq N$. Then $G'$ is still Borel and has bounded degree $d + d^2 + \cdots + d^N$. By Proposition 4.5 in \cite{KST}, $\chi_B(G')$ is countable, so by Proposition 4.2 in \cite{KST}, there is a Borel maximal $(1,G')$-discrete subset $A \subset X$. Note $(1,G')$-discrete sets are exactly $(N,G)$-discrete sets, though.
\end{proof}

Our marked groups are as follows: Let $\Gamma = (\Z/3\Z \times \Z, \{(1,0),(2,0),(0,\pm1)\})$, and let $\Delta = (\Z/3\Z \rtimes_{\varphi} \Z, \{(1,0),(2,0),(0,\pm1)\})$, where $\varphi:\Z \rightarrow \textnormal{Aut}(\Z/3\Z)$ is the homomorphism sending $1$ to the inversion map. An isomorphism $\cay(\Gamma) \rightarrow \cay(\Delta)$ is given by 
\begin{equation}
(i,n) \mapsto ((-1)^n i,n).
\end{equation}
Clearly the Cayley graphs have chromatic number 3.

\begin{theorem}\label{th:3}
$\chi_B(F(2^\Gamma)) = 3$
\end{theorem}
\begin{proof}
By Lemma \ref{discrete}, let $A$ be a maximal Borel independent subset of the shift graph. Let $E$ be some $\Z/3\Z$ orbit. $E$ can contain at most one point in $A$. Suppose $E$ contains no points in $A$. Then by maximality, for each $x \in E$, either $(0,1) \cdot x$ or $(0,-1) \cdot x$ is in $A$. Then by the pidgeonhole principle, either $(0,1) \cdot E$ or $(0,-1) \cdot E$ contains two points of $A$, contradicting independence. Therefore $A$ contains exactly one element from each $\Z/3\Z$ orbit. One can check that the action of a central element on a shift graph preserves the graph relation, so $(1,0) \cdot A$ and $(2,0) \cdot A$ are still Borel independent sets. Since these sets along with $A$ partition $F(2^\Gamma)$, they constitute a Borel 3-coloring.
\end{proof}

When working with $\Delta$, we can obtain a Borel set having 1 element in each $\Z/3\Z$ orbit in the same way, but the action of $(1,0)$ no longer preserves the graph relation, so we cannot use the above strategy to get a 3-coloring from it. In fact, we cannot use any strategy at all:

\begin{theorem}\label{th:4}
$\chi_B(F(2^\Delta)) = 4$
\end{theorem}
\begin{proof}
Since $F(2^\Delta)$ has bounded degree 4, the upper bound $\chi_B(F(\N^\Delta)) \leq 4$ follows from 5.12 in \cite{CK}.

For the lower bound, suppose we have some 3-coloring $c:F(2^\Delta) \rightarrow \{1,2,3\}$. Each $\Z/3\Z$ orbit $E$ must use all 3 colors, so the action of $(1,0)$ on it must induce a permutation on the colors which is a 3-cycle in $S_3$: either $(1\ 2\ 3)$ or $(1\ 3\ 2)$. One can easily check that whichever holds for $E$, the other must hold for its neighbors $(0,\pm 1) \cdot E$. Therefore, by considering which holds for each orbit, we can get a 2-coloring of the $\Z/3\Z$ orbits of $F(2^\Delta)$, that is, a $\Z/3\Z$-invariant map $f:F(2^\Delta) \rightarrow \{1,2\}$ with the property that $f(x) \neq f((0,1) \cdot x)$ for all $x$.

Now, give $\Z$ its usual generators and let $g:F(2^\Z) \rightarrow F(2^{\Delta})$ be the map given by 
\begin{equation}
g(x)(i,n) = \begin{cases} x(n) \  \textnormal{ if } i = 0  \\ 0 \  \textnormal{ else.} \end{cases}
\end{equation}
$g$ is clearly continuous and satisfies $g(1 \cdot x) = (0,1) \cdot g(x)$. Therefore $f \circ g$ is a 2-coloring of $F(2^\Z)$. If $c$ is Borel, then $f$ is Borel, so $f \circ g$ is Borel. It is well known that $\chi_B(F(2^\Z)) = 3$, though (see for example \cite{KST}), so $f \circ g$ cannot be Borel. Therefore $c$ cannot be Borel.
\end{proof}

As a Corollary of these calculations, Problem \ref{prob:main} has a negative answer.


\section{Two Proofs of Theorem \ref{th:main}}\label{sec:big}


We now construct two infinite families of pairs of Marked groups $\Gamma_k$ and $\Delta_k$ witnessing Theorem \ref{th:main}. The first is based on and builds on the ideas from Section \ref{sec:ex}. The second takes a different approach, but is simpler overall and demonstrates that (at least one of) the groups involved can be torsion free.

\subsection{Semi-direct Products of $\Z/k\Z$ and $\Z$}

Fix $k \geq 3$. For the first family, let $\Gamma_k = \Z/k\Z \times \Z$ and $\Delta_k = \Z/k\Z \rtimes_{\varphi_k} \Z$, where $\varphi: \Z \rightarrow \textnormal{Aut}(\Z/k\Z)$ is the homomorphism sending 1 to the inversion map. The fixed generators are as follows: $i$ will always represent an arbitrary element of $\Z/k\Z$. For each group, we use the elements represented by the ordered pairs $(i,0)$ for all $i \neq 0$ and $(i,1)$ for all $i \neq 0,1$. We also include the inverses of all these elements to make the generating set symmetric. For $\Gamma_k$, this requires adding the elements $(-i,-1)$ for all $i \neq 0,1$. For $\Delta_k$, this requires adding the elements $(i,-1)$ for all $i \neq 0,1$. Thus, for each group's Cayley graph, the induced graph on each $\Z/k\Z$ orbit $E$ is the complete graph on $k$ vertices. Furthermore, for the Cayley graph on $\Gamma_k$, $(i,n)$ and $(j,n+1)$ are adjacent if and only if $j-i \neq 0,1$. For $\Delta_k$, they are adjacent if and only if $i+j \neq 0,1$.

We'll now construct an isomorphism, $\phi$, from the Cayley graph of $\Gamma_k$ to the Cayley graph of $\Delta_k$. Firstly, for even $n = 2m$, we set $\phi(i,n) = (i-m,n)$. Then, for $n$ odd, we set $\phi(i,n) = (-j,n)$, where $\phi(i,n+1) = (j,n+1)$. Since $\phi$ preserves the $\Z/k\Z$ orbits, we only need to check the edge relations between neighboring $\Z/k\Z$ orbits $E$ and $(0,1) \cdot E$ to check that $\phi$ is an isomorphism. Fix some even $n = 2m$. In the Cayley graph of $\Gamma_k$, $(i,n)$ and $(j,n-1)$ are connected exactly when $j \neq i,i-1$. Meanwhile, $\phi$ sends these points to $(i-m,n)$ and $(m-j,n-1)$, respectively, and these are connected exactly when $(i-m) + (m-j) = i-j \neq 0,1$, which is equivalent to our earlier condition. Similarly, in the Cayley graph of $\Gamma_k$, $(i,n),(j,n+1)$ are connected exactly when $j \neq i,i+1$, and $\phi$ sends these points to $(i-m,n)$ and $(m+1-j,n+1)$, respectively, and these are connected exactly when $(i-m)+(m+1-j)=1+i-j \neq 0,1$, which is equivalent to our earlier condition. This isomorphism is demonstrated in Figure $\ref{fig:iso}$. It can be realized by overlaying the drawings in parts (a) and (b) of the figure.

\begin{figure}
\centering
\includegraphics[width=1\textwidth]{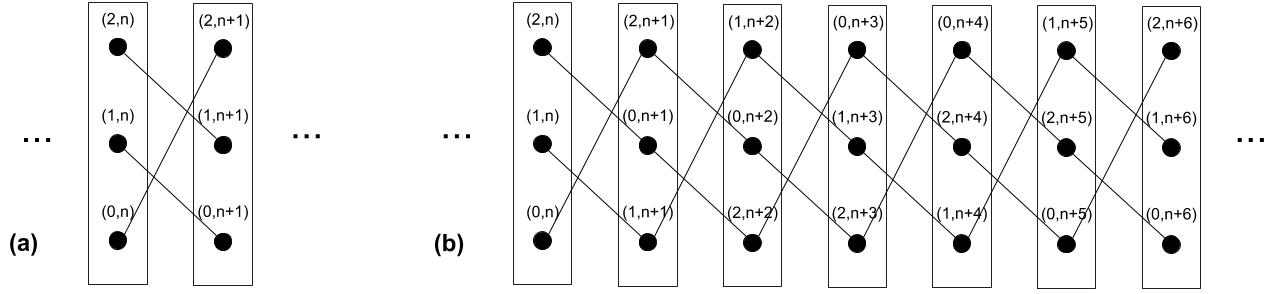}
\caption{\label{fig:iso} (a): A drawing of a segment of the Cayley graph for $\Gamma_3$. Rectangles represent sets on which the induced subgraph is complete, so the edges within them are not drawn. (b): Same as (a), but for $\Delta_3$, and realized using the isomorphism $\phi$.}
\end{figure}

These Cayley graphs clearly have chromatic number $k$. It will be useful to understand how a $k$-coloring on them must look. Suppose we have such a coloring $c$ using the colors 1, 2, $\ldots,$ $k$.  On any $\Z/k\Z$ orbit $E$, all $k$ colors must be used. Let $\tau_E \in S_k$ be the permutation of the colors induced by the action of $(1,0)$ on $E$. Clearly $\tau_E$ is always a $k$-cycle. Fix some $E$, and relabel the colors if necessary so that $\tau_E = (1\ 2\ \cdots\ k)$. Then the restriction of $c$ to $E$ must look like $c(i,n) = i+m$ (mod $k$) for some fixed $m$. Now, suppose we are working in $\Gamma_k$. The neighboring orbit $(0,1) \cdot E$ must again use all $k$ colors, and for each $i$, we must have $c(i,n+1) = c(i,n)$ or $c(i-1,n)$, and so $c(i,n+1) = i+m-\epsilon_i$, where $\epsilon_i = 0$ or 1 for each $i$. Since each color must be used exactly once, we must have either $\epsilon_i = 0$ for all $i$ or $\epsilon_i = 1$ for all $i$.

In either case, $\tau_{(0,1) \cdot E} = \tau_E$, and so we see this holds in general. On the other hand, comparing this situation to the one in $\Delta_k$, we see that for $\Delta_k$, $\tau_{(0,1) \cdot E} = \tau_E\inv$. Note the similarity here to the one in the proof of Theorem \ref{th:4}. Clearly these equations still hold when $E$ is a $\Z/k\Z$ orbit in one of the shift graphs.

We may now compute our chromatic numbers:

\begin{theorem}\label{th:ab}
$\chi_B(F(2^{\Gamma_k})) = k$.
\end{theorem}
\begin{proof}

By Lemma \ref{discrete}, let $A \subset F(2^{\Gamma_k})$ be a Borel maximal $k$-discrete set. Then if $E_1$ and $E_2$ are two different $\Z/k\Z$ orbits in the same connected component meeting $A$, $\rho(E_1,E_2) \geq k-1$. Furthermore, such orbits occur infinitely often in each $\Z$-direction in the following sense: Given some $\Z/k\Z$ orbit $E$, there are positive $N_1,N_2$ such that $(0,N_1) \cdot E$ and $(0,-N_2) \cdot E$ both meet $A$.

We now extend $A$ to a Borel independent set $A'$ which meets each $\Z/k\Z$ orbit exactly once. Let $x \in A$ with $\Z/k\Z$ orbit $E$ and $N$ the smallest positive integer such that $(0,N) \cdot E$ contains a point of $A$. Call that point $y$. Then there is a unique $0 \leq i < k$ such that $y = (i,N) \cdot x$. Now simply add the points $(1,1) \cdot x$, $(2,2) \cdot x, \ldots, (i,i) \cdot x, (i,i+1) \cdot x, \ldots, (i,N-1) \cdot x$ to $A'$. This will always work since $N \geq k-1$ by our definition of $A$.

Then, as in the proof of Theorem \ref{th:3}, $\bigsqcup_{i=0}^{k-1} (i,0) \cdot A'$ gives a Borel $k$-coloring of $F(2^{\Gamma_k}))$.
\end{proof}

\begin{theorem}\label{th:non}
$\chi_B(F(2^{\Delta_k})) = k+1$
\end{theorem}

\begin{proof}
We start by showing the Borel chromatic number must be greater than $k$. Suppose we have a $k$-coloring $c: F(2^{\Delta_k}) \rightarrow \{1,2,\ldots,k\}$. Let $C$ be the set of $k$-cycles in $S_k$, and let $C = C_1 \sqcup C_2$ be a partition of $C$ such that $C_2 = C_1\inv$. This exists because $k \geq 3$, so no $k$-cycle is its own inverse. Now define $f:\chi_B(F(2^{\Delta_k})) \rightarrow \{1,2\}$ by sending $x \mapsto i$ exactly when $\tau_E \in C_i$, where $E$ is the $\Z/k\Z$ orbit of $x$. By the remarks preceding Theorem \ref{th:ab}, $f$ is a 2-coloring of the $\Z/k\Z$ orbits of $F(2^{\Delta_k})$. This gives a contradiction to $\chi_B(F(2^\Z)) = 3$ as in the proof of Theorem \ref{th:4}.

We next construct a Borel $(k+1)$-coloring using the colors $1,2,\ldots,k+1$. Let $A \subset F(2^{\Delta_k})$ be a Borel independent set as in the proof of Theorem $\ref{th:ab}$, except with the path distances between different $\Z/k\Z$ orbits meeting $A$ all at least $3(k-1)$. For each $x \in A$, color the $\Z/k\Z$ orbit of $x$ by setting the color of $(i,0) \cdot x$ to be $i+1$ for $0 \leq i \leq k-1$.

We now color the $\Z/k\Z$ orbits between those meeting $A$. Let $x \in A$ with $\Z/k\Z$ orbit $E$ and $N$ the smallest positive integer such that $(0,N) \cdot E$ contains a point of $A$. Let $y = (0,N) \cdot x$. We will color our $\Z/k\Z$ orbits between $x$ and $y$ in a way such that $(0,3n) \cdot E$ uses only the colors 1 through $k$ for each non-negative $n < N/3$. For every such $n$, let $\sigma_n \in S_k$ be the permutation defined by $c(((-1)^n (i-1),3n) \cdot x) = \sigma_n(i)$, where $c$ is our coloring. For example, $\sigma_0 = \id$. Let $\sigma' \in S_k$ be defined by $c(((-1)^N (i-1),0) \cdot y) = \sigma'(i)$. For every possible such $\sigma'$, fix a way of writing $\sigma'$ as a product of transpositions $\sigma' = \rho_l\rho_{l-1} \cdots \rho_1$ with $l \leq k-1$. We can now arrange so that for $0 \leq n \leq l$, $\sigma_n = \rho_n \rho_{n-1} \cdots \rho_1$. We do this inductively as follows: 

If $E'$ is a $\Z/k\Z$ orbit colored by $c$, then by $c_{\textnormal{shift}}(E')$ shall be meant the coloring on $(0,1) \cdot E'$ given by $c_{\textnormal{shift}}(E')(z) = c((0,-1) \cdot z)$. Given a coloring on $(0,3n) \cdot E$ satisfying our condition, and $\rho_{n+1} = (a\ b)$ with $1 \leq a < b \leq k$, color $(0,3n+1) \cdot E$ with $c_{\textnormal{shift}}((0,3n) \cdot E)$, but then swap the color $a$ with the color $k+1$. Then color $(0,3n+2) \cdot E$ with $c_{\textnormal{shift}}((0,3n+1) \cdot E)$, but swap the color $b$ with the color $a$. Finally color $(0,3n+3) \cdot E$ with $c_{\textnormal{shift}}((0,3n+2) \cdot E)$, but swap the color $k+1$ with the color $b$. This can be best understood by looking at the example in Figure \ref{fig:swap}.

\begin{figure}
\centering
\includegraphics[width=0.4\textwidth]{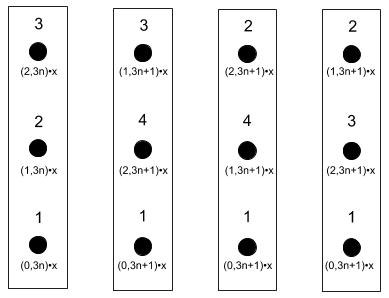}
\caption{\label{fig:swap} An example of the technique for going from $\sigma_n$ to $\sigma_{n+1} = \rho_{n+1} \sigma_n$. Here $k=3$, $\sigma_n = \id$, and $\rho_{n+1} = (2\ 3)$. The numbers over each point are the colors. The rectangles have the same meaning as in Figure \ref{fig:iso}. Edges between $\Z/k\Z$ orbits are omitted. The action of $(0,1)$ sends each point to its immediate neighbor on the right.}
\end{figure}

Thus all the $\Z/k\Z$ orbits between $E$ and $(0,3l) \cdot E$ are colored, and $\sigma_{l} = \sigma'$. Note that $3l \leq N$ by the definition of $A$. Now, for $3l < n < N$, color the orbits $(0,n) \cdot E$ inductively by using $c_{\textnormal{shift}}((0,n-1) \cdot E)$ for each. This completes our Borel coloring. 
\end{proof}

\subsection{$\Z$ versus $\Z/2\Z * \Z/2\Z$}\label{sec:big2}

For the second family, one of the underlying groups will be the free product $\Z/2\Z * \Z/2\Z$. When working with this group, we will always use $a$ and $b$ to represent the natural generators of it. That is,
\begin{equation}
\Z/2\Z * \Z/2\Z = \langle a,b \mid a^2=b^2=1 \rangle.
\end{equation}
Let $\Gamma_k' = (\Z,\{\pm 2, \pm 3, \ldots, \pm 2k - 3\})$, and let $\Delta_k' = (\Z/2\Z * \Z/2\Z, \{ab,ba,$ $aba,bab,\ldots,(ab)^{k-2}a,(ba)^{k-2}b \} )$. The Cayley graphs of $\Gamma_k'$ and $\Delta_k'$ are clearly isomorphic.

Let us start by examining an arbitrary coloring $c$ of $\cay(\Gamma_k')$. There must be some $n \in \Z$ such that $c(n) \neq c(n+1)$. Then each integer $n+1 \leq m \leq 2k-3$ must have a different color from $n$. Furthermore, no more than two integers in this interval can share the same color, and this can occur only if their positive difference is 1. There are $2k-3$ integers in this interval, so this interval must use at least $k-1$ colors, all different from $c(n)$. Therefore $\chi(\cay(\Gamma_k')) = \chi(\cay(\Delta_k')) \geq k$. From Theorem \ref{th:ab2} it will follow that the chromatic numbers are exactly $k$.

Suppose now that $c$ is a $k$-coloring using the colors $1,2,\ldots,k$, and that $c(n) = 1$. Then our interval above must use all of the colors besides 1, and it must use all but one of them twice. Define a function $\delta:\Z \rightarrow \Z/k\Z$ by $\delta(m) = c(m') - c(m)$ where $m' \in \Z$ is the least integer greater than $m$ with $c(m') \neq c(m)$ (so $m' = m+1$ or $m+2$). Clearly, we can relabel colors such that $\delta(m) = 1$ for all $n \leq m \leq 2k-3$. A simple inductive argument then shows that $\delta(m) = 1$ for all $m \in \Z$.

We may now again compute some chromatic numbers:

\begin{theorem}\label{th:ab2}
$\chi_B(F(2^{\Gamma_k'})) = k$.
\end{theorem}
\begin{proof}
Using Lemma \ref{discrete}, let $A \subset F(2^{\Gamma_k'})$ be a maximal Borel $((2k-1)^2,F(2^\Z))$-discrete set where $\Z$ refers to the marked group $(\Z,\{\pm 1\})$. Give each point in $A$ the color 1. We now color the points in between those of $A$.

Fix an $x \in A$, and let $N$ be the smallest positive integer such that $N \cdot x \in A$. Let $0 \leq l < 2k$ such that $N = -l$ (mod $2k$). Now, for each $0 \leq n \leq l$, give $(2k-1)n \cdot x$ the color 1. Let $M = (2k-1)l$. Note that $N - M = 0$ (mod $2k$) and $M \leq N$. Now for each $0 \leq n < (M-N)/2k$, give both $(2kn+M) \cdot x$ and $(2kn+1+M) \cdot x$ the color 1.

Let $A'$ be the set of everything we have colored 1 so far. Note that if $y \in A'$ is such that $1 \cdot y \not\in A'$ and $N$ is the smallest positive integer such that $N \cdot y \in A'$, then $N = 2k-1$. Now, for the rest of the points $y \in F(2^{\Gamma_k'}) \setminus A$, if $N$ is the smallest positive integer such that $(-N) \cdot y \in A'$, give $y$ the color $\lceil N/2 \rceil+1$. It is clear that the smallest possible $\rho_{F(2^\Z)}$ distance between two points of the same color (besides the distance 1) is $2k-2$, so this constitutes a Borel $k$-coloring.
\end{proof}

Note the crucial role played by the natural Borel $\Z$-ordering on the connected components of $F(2^\Z)$ given by $x < y$ if and only if there exists $N > 0$ such that $N \cdot x = y$. It allowed us to count `up' from one element of $A$ to the next. Our proof that $\chi_B(F(2^{\Delta_k'})) > k$ will rely on the following fact from \cite{KST} (see page 19), which makes precise the idea that there is no such way of orienting ourselves when working with $F(2^{\Z/2\Z * \Z/2\Z})$.

\begin{lemma}\label{order}
There is no Borel relation $<$ on $F(2^{\Z/2\Z * \Z/2\Z})$ such that for each connected component $E$ of $F(2^{\Z/2\Z * \Z/2\Z})$, the restriction of $<$ to $E$ is a linear order of type $\Z$ such that for each $x \in E$, the successor and predecessor of $x$ under this order are (not necessarily respectively) $a \cdot x$ and $b \cdot x$.
\end{lemma}

If $\prec$ is a $\Z$-ordering of some connected component $E$ with the property described above, call it a \textit{natural} $\Z$-ordering of $E$. Note that every $E$ has exactly two choices of natural orderings. If $\prec$ is one, call the other $\prec\inv$. It is of course given by $x \prec\inv y$ if and only if $y \prec x$.

\begin{theorem}\label{th:non2}
$\chi_B(F(2^{\Delta_k'})) = k+1$
\end{theorem}
\begin{proof}
We start by showing the Borel chromatic number must be greater than $k$. Suppose, to the contrary, we have a Borel $k$-coloring $c: F(2^{\Delta_k'}) \rightarrow \{1,2,\ldots,k\}$. Let $C$ be the set of $k$-cycles in $S_k$, and let $C = C_1 \sqcup C_2$ be a partition as in the proof of Theorem \ref{th:non}. 

Let $E$ be a connected component of our shift graph. If we have a natural $\Z$-ordering $\prec$ on $E$, then by the remarks preceding Theorem \ref{th:ab2} the function
\begin{equation}
\begin{split}
i \mapsto c(x') \textnormal{ such that } x \in E \textnormal{ with } c(x) = i \ \ \ \ \ \\ \textnormal{ and } x' \textnormal{ is the least point } \succ x \textnormal{ with } c(x') \neq c(x) 
\end{split}
\end{equation}
is well defined and defines a $k$-cycle in $S_k$, call it $\tau_{E,\prec}$. Note that $\tau_{E,\prec\inv} = \tau_{E,\prec}\inv$.

By Lemma \ref{order}, we cannot choose a natural $\Z$-ordering for each $E$ in a Borel way, but we can still determine the set $\{\tau_{E,\prec},\tau_{E,\prec\inv}\}$. By always choosing the representative of this set in $C_1$, we get a Borel $\Delta_k'$-invariant map $f:F(2^{\Delta_k'}) \rightarrow C_1$ such that $f(E) \subset \{\tau_{E,\prec},\tau_{E,\prec\inv}\}$ for each $E$. This allows us to contradict Lemma \ref{order}, though: From $f$, we can define a Borel relation $<$ contradicting the lemma by setting $x < y$ if and only if $x$ and $y$ are in the same connected component $E$, $x \neq y$, and $\tau_{E,\prec} = f(E)$ for the unique choice of $\prec$ satisfying $x \prec y$.

We next construct a Borel $(k+1)$ coloring, $c$, using the colors $1,2,\ldots,k+1$. As in the proof of Theorem \ref{th:ab2}, let $A \subset F(2^{\Delta_k'})$ be a maximal Borel $(5k^2,F(2^{\Z/2\Z * \Z/2\Z}))$-discrete set where $\Z/2\Z * \Z/2\Z$ refers to the marked group $(\Z/2\Z * \Z/2\Z,\{a,b\})$. First give each point of $A$ the color 1. Then color the points around each point of $A$ as follows: Let $x \in A$, and let $E$ be the connected component of $x$. Let $\prec_x$ be the natural $\Z$-ordering on $E$ satisfying $x \prec_x a \cdot x$. Every $\Z$-ordering $\prec$ on $E$ induces a $\Z$-action, call it $\cdot_{\prec}$, on $E$, determined by $1 \cdot_{\prec} x = y$ such $y$ is the successor of $x$ under $\prec$. For each $0 < n \leq 2(k-1)$, give $n \cdot_{\prec_x} x$ the color $\lceil n/2 \rceil + 1$ and $(-n) \cdot_{\prec_x} x$ the color $k+2-c(n \cdot_{\prec_x} x)$. We now color the points in between these clusters.

\begin{figure}
\centering
\includegraphics[width=1\textwidth]{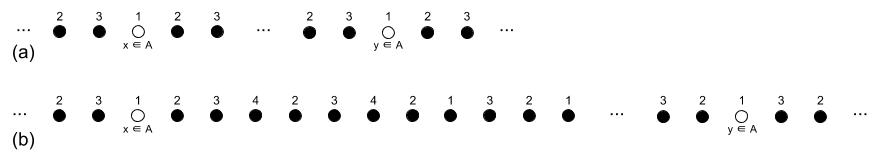}
\caption{\label{fig:lin} A drawing of the coloring procedure for $\Delta_3'$. Points are arranged in a line on the page according to some natural $\Z$-order for their connected component. Hollow points represent a single point, while filled in ones represent two adjacent (in the $\Z$-order) ones having the same color. The numbers over the points denote the colors. $x$ and $y$ are as in the proof of the Theorem. (a): The case where $\prec_x = \prec_y$. (b): The case where $\prec_x \neq \prec_y$.}
\end{figure}

Let $x,y \in A$ be distinct points in the same connected component $E$ such that for any natural $\Z$-ordering on $E$, there are no points in between $x$ and $y$. Suppose we have $\prec_x = \prec_y$. This situation is depicted in part (a) of Figure \ref{fig:lin}. Then, following the technique of the proof of Theorem \ref{th:ab2}, we can use the colors $1$ through $k$ to color all the points represented by the middle elipses in the Figure. Note that $x$ and $y$ are far enough apart to allow for this.

The difficulty occurs when $\prec_x \neq \prec_y$, as depicted in part (b) of the Figure. Fix for now a natural $\Z$-ordering, $\prec$, on $E$. Using $\prec$, we can specify the following procedure for coloring the points between $x$ and $y$: Suppose, without loss of generality, $x \prec y$. Also assume $\prec = \prec_x$: If not, the procedure is essentially identical. Let $N = 6k \lfloor k/2 \rfloor$. We start by coloring the points $n \cdot_{\prec} x$ for $n$ even and $0 \leq n < N+2k$. Given such an $n$, consider the points $n \cdot_\prec x, (n+2) \cdot_\prec x, \ldots, (n+2(k-1)) \cdot_\prec x$, and let $\sigma_n \in \{1,2,\ldots,k+1\}^k$ be the sequence of colors they take respectively. For example, $\sigma_0 = (1,2,3,\ldots,k)$.

Our goal is to arrange things so that $\sigma_{6k} = (k,2,3,\ldots,k-1,1)$, $\sigma_{12k} = (k,k-1,3,\ldots,k-2,2,1)$, and so on to $\sigma_{N} = (k,\ldots,2,1)$. Note that our desired $\sigma_{6km}$ and $\sigma_{6k(m+1)}$ differ by a single transposition. Therefore we can use the procedure seen in part (b) of Figure \ref{fig:lin}, which mirrors that seen in Figure \ref{fig:swap}: Given $\sigma_{6km}$ and the transposition $(c\ d)$ relating $\sigma_{6km}$ and our goal for $\sigma_{6k(m+1)}$, where $c < d$, color the succeeding points such that $\sigma_{6km + 2k} = \sigma_{6km}$ but with $c$ and $k+1$ switched, then $\sigma_{6km + 4k} = \sigma_{6km + 2k}$ but with $d$ and $c$ switched, and $\sigma_{6km + 6k} = \sigma_{6km + 4k}$ but with $k+1$ and $d$ switched.

Next, for $0 < n < N + 2k$ with $n$ odd, simply set $c(n \cdot_\prec x) = c((n+1) \cdot_\prec x)$. When all this is done, we will be `aligned' with the coloring we've established around $y$, and can again use the technique from the proof of Theorem $\ref{th:ab2}$ to color in the points between $(N+2k) \cdot_\prec x$ and $y$ using the colors $1$ through $k$.

Because of Lemma \ref{order}, it may seem that the dependence of the above procedure on a choice of $\prec$ renders it useless. However, this is not so since we have the ability to choose a different $\prec$ for each pair of points $x,y$ as above with $\prec_x \neq \prec_y$. To do this in a Borel fashion, it is enough to pick out one of $x$ or $y$ for every such pair in a Borel fashion, since then we can set $\prec = \prec_x$ or $\prec_y$ according to which one was picked. Thus, thanks to the following Lemma, we are done:
\end{proof}

\begin{lemma}\label{choice}
Let $X$ be a standard Borel space, and $[X]^{<\infty}$ the standard Borel space of finite subsets of $X$. There is a Borel choice function $p:[X]^{<\infty} \setminus \{\emptyset\} \rightarrow X$.
\end{lemma}
\begin{proof}
Let $P \subset [X]^{<\infty} \times X$ be the set of pairs $(S,x)$ such that $x \in S$. $P$ is Borel, so the result follows immediately from the Lusin-Novikov Uniformization Theorem (see, for example, Theorem 18.10 in \cite{K95}.)
\end{proof}


\section{Borel Isomorphic Equivalence Relations}\label{sec:iso}


Let $X$ be a standard Borel space and $E$ an equivalence relation on $X$. We call $E$ \textit{Borel} if it is Borel as a subset of $X \times X$. We will be particularly interested in the case where $E$ is the equivalence relation generated by some Borel graph $G$, so that the equivalence classes of $E$ are by definition the connected components of $G$. If $F$ is another Borel equivalence relation on another standard Borel space $Y$, we say $E$ and $F$ are \textit{Borel isomorphic} if there is a Borel bijection $f:X \rightarrow Y$ such that $(x,y) \in E$ if and only if $(f(x),f(y)) \in F$. In this section we prove our claim from Section \ref{sec:intro} that for each of our pairs $\Gamma_k$ and $\Delta_k$, the equivalence relations generated by $F(2^{\Gamma_k})$ and $F(2^{\Delta_k})$ are Borel isomorphic.

Our proof will depend on Theorem 2' from \cite{DJK}, which will require a few definitions to state. 

A relation $E$ as above is called \textit{smooth} if there is a Borel subset $A \subset X$ meeting each $E$ equivalence class exactly once. It is called \textit{aperiodic} if it its equivalence classes are all infinite.

Suppose a relation $E$ as above is induced by a Borel action of a countable group $\Gamma$ on $X$. Let $\mu$ be a Borel probability measure on $X$. We call $\mu$ \textit{$E$-invariant} if it $\Gamma$-invariant, that is, if $\mu(g \cdot A) = \mu(A)$ for any $g \in \Gamma$ and $\mu$-measurable $A \subset X$. This definition does not depend on the choice of action or group. We call $\mu$ \textit{$E$-ergodic} if every $E$-invariant Borel set has measure 0 or 1. Finally, a Borel equivalence relation $E$ is called \textit{hyperfinite} if it can be induced by a Borel action of $\Z$. A Borel action of $\Gamma$ on $X$ as above is called hyperfinite if the equivalence relation it induces is hyperfinite.

We can now state Theorem 2' from \cite{DJK}:
\begin{theorem}\label{th:iso}
For a Borel equivalence relation $E$ on a standard Borel space $X$, let $C(E)$ denote the number of $E$-invariant, $E$-ergodic Borel probability measures on $X$. Two aperiodic, non smooth hyperfintie Borel equivalence relations $E$ and $F$ are Borel isomorphic if and only if $C(E) = C(F)$.
\end{theorem}

The following two Lemmas pointed out to us by A.S. Kechris will allow us to apply this theorem to our situation:

\begin{lemma}\label{lem:virtual}
Let $\Gamma$ be a virtually infinite cyclic group, that is, a group with a finite index subgroup isomorphic to $\Z$. Then every Borel action of $\Gamma$ is hyperfinite
\end{lemma}
\begin{proof}
Fix a Borel action of $\Gamma$ on some standard Borel space $X$. Let $E$ be the equivalence relation induced by the action of $\Gamma$ and $F$ the equivalence relation induced by the action of the finite index subgroup isomorphic to $\Z$. Let the index be $n$. Now, $F$ is hyperfinite by definition, $F \subset E$, and each $E$ equivalence class contains at most $n$ $F$ equivalence classes. Proposition 1.3(vii) in \cite{JKL} then tells us $E$ is hyperfinite.
\end{proof}

\begin{lemma}\label{lem:measure}
Let $\Gamma$ be a marked group. Let $E$ be the connected component equivalence relation for $F(2^\Gamma)$. Then $C(E) = 2^{\aleph_0}$.
\end{lemma}
\begin{proof}
$C(E)$ is at most $2^{\aleph_0}$ for any $E$ (see \cite{DJK}), so it suffices to demonstrate a distinct $E$-invariant, $E$-ergodic Borel probability measure $\mu_p$ for each $p \in (0,1)$.

For each such $p$, let $\mu_p'$ be the probability measure on the discrete space 2 defined by $\mu_p'(\{0\}) = p$. Then let $\mu_p$ be the product measure on $2^{\Gamma}$ coming from $\mu_p'$. $\mu_p$ can be equivalently defined as follows: It is the unique Borel probability measure on $2^\Gamma$ such that for every partial function $g:\Gamma \rightarrow 2$ with finite domain, it assigns the clopen set $\{x \mid x|_{\textnormal{dom}(g)} = g \}$ a measure of $p^{|g\inv(0)|}(1-p)^{|g\inv(1)|}$. Each $\mu_p$ is $E$-invariant and $E$-ergodic, and they are clearly all distinct.
\end{proof}

The groups $\Z$ and $\Z/2\Z * \Z/2\Z$ are clearly virtually infinite cyclic, as are all groups of the form $\Z/k\Z \rtimes \Z$ for all $k \neq 0$. Furthermore, the connected component equivalence relation on $F(2^\Gamma)$ for any infinite $\Gamma$ is aperiodic by definition, and since the measures $\mu_p$ in the proof of Lemma \ref{lem:measure} give a measure of 0 to each singleton set, it is also non smooth for every $\Gamma$ by Theorem 3.4 in \cite{DJK}.  Therefore, combining Lemmas \ref{lem:virtual} and \ref{lem:measure} with Theorem \ref{th:iso} gives us the main claim of this section, which we record below:

\begin{theorem}\label{th:mainiso}
Let $k \geq 3$. The marked groups $\Gamma_k$ and $\Delta_k$ appearing in Theorem \ref{th:main} can be chosen such that the connected component equivalence relations on $F(2^{\Gamma_k})$ and $F(2^{\Delta_k})$ are Borel isomorphic.
\end{theorem}


\section{Measurable Chromatic Numbers}\label{sec:measure}


Let $X$ be a Polish space and $G$ a Borel graph on it. Let $\mu$ be a Borel probability measure on $X$. Instead of Borel colorings, one sometimes considers \textit{$\mu$-measurable} colorings or \textit{Baire measurable} colorings, defined in the obvious ways. Let $\chi_\mu(G)$ and $\chi_{BM}(G)$ denote the \textit{$\mu$-measurable} and \textit{Baire measurable} chromatic numbers of $G$, respectively, defined in the obvious ways. We define the \textit{measure chromatic number} of $G$, denoted $\chi_M(G)$, as the supremum of $\chi_\mu(G)$ taken over all $\mu$.

Furthermore, for $\mu$ as above and $\epsilon > 0$, we can consider \textit{$\epsilon$-approximate $\mu$-measurable} colorings: Measurable partial functions $c$ with $\textnormal{dom}(c)$ Borel and $\mu(X \setminus \textnormal{dom}(c)) < \epsilon$. We define the \textit{approximate $\mu$-measurable chromatic number} of $G$, denoted $\chi_\mu^{ap}(G)$, as the smallest $k$ such that $G$ has an $\epsilon$-approximate $\mu$-measurable coloring for each $\epsilon > 0$. We define the \textit{approximate measure chromatic number} of $G$, denoted $\chi_M^{ap}$, as the supremum of $\chi_\mu^{ap}(G)$ taken over all $\mu$.

In \cite{KM}, alongside Problem \ref{prob:main}, the analogues of Problem 1 for measurable and approximately measurable chromatic numbers are posed. Note that $\chi(G) \leq \chi_M(G),\chi_{BM}(G) \leq \chi_B(G)$ for any $G$. Furthermore it's still the case that $\chi_{BM}(F(2^\Z)) = \chi_M(F(2^\Z)) = 3$, and one can easily check that the analogue of Lemma \ref{order} still holds if we ask that $<$ be a Baire measurable or $\mu$-measurable relation for any $\mu$. Therefore, the theorems of Section $\ref{sec:big}$ and their proofs give us the following extension of Theorem \ref{th:main}, resolving some of the above problems:

\begin{theorem}
Let $k \geq 3$. There exist marked groups $\Gamma_k$ and $\Delta_k$ with isomorphic Cayley graphs but for which \begin{enumerate}[label=(\roman*)]
	\item $\chi_B(F(2^{\Gamma_k})) = \chi_{BM}(F(2^{\Gamma_k}))  = \chi_{M}(F(2^{\Gamma_k})) = \chi(F(2^{\Gamma_k})) = \chi(F(2^{\Delta_k})) = \chi(\cay(\Gamma_k)) = \chi(\cay(\Delta_k)) = k,$  
	\item $\chi_B(F(2^{\Delta_k})) = \chi_{BM}(F(2^{\Delta_k})) = \chi_{M}(F(2^{\Delta_k})) = k+1$.
\end{enumerate}
\end{theorem}

However, the approximately measurable version of Problem \ref{prob:main} seems to still be open:

\begin{prob}\label{prob:approx}
Let $\Gamma$ and $\Delta$ be two marked groups with isomorphic Cayley graphs. Must we have $\chi_M^{ap}(F(2^\Gamma)) = \chi_M^{ap}(F(2^\Delta))$?
\end{prob}


\section{Further Problems}\label{sec:probs}


In this final section, we consider some natural questions that arise in light of Theorem \ref{th:main}.

\subsection{Larger Differences in Borel Chromatic Numbers}

Perhaps the most natural such question is the following:
\begin{prob}\label{prob:big}
Let $\Gamma$ and $\Delta$ be two marked groups with isomorphic Cayley graphs. Must we have $$\left| \chi_B(F(2^\Gamma)) - \chi_B(F(2^\Delta)) \right| \leq 1 ?$$
\end{prob}
This question has a connection to one posed by Marks: In \cite{M15}, he proves
\begin{theorem}\label{th:marks}
Let $\Gamma$ and $\Delta$ be two marked groups. Then
\begin{equation}
\chi_B(F(2^{\Gamma * \Delta})) \geq \chi_B(F(2^{\Gamma})) + \chi_B(F(2^\Delta)) - 1,
\end{equation}
where the generating set of $\Gamma * \Delta$ is the union of those of $\Gamma$ and $\Delta$.
\end{theorem}
However, it is open whether or not this bound can ever be exceeded:
\begin{prob}\label{prob:marks}
Let $\Gamma$ and $\Delta$ be two marked groups. Must we have
\begin{equation}
\chi_B(F(2^{\Gamma * \Delta})) = \chi_B(F(2^{\Gamma})) + \chi_B(F(2^\Delta)) - 1?
\end{equation}
\end{prob}
Now, fix a $k \geq 3$ and let $\Gamma_k$ and $\Delta_k$ be two groups witnessing Theorem \ref{th:main} for this value of $k$. Then $\Gamma_k * \Gamma_k$ and $\Delta_k * \Delta_k$ still have isomorphic Cayley graphs, but the lower bound of Theorem \ref{th:marks} is $2k-1$ for the former and $2k+1$ for the latter. Therefore, if the shift graph of $\Gamma_k * \Gamma_k$ does not exceed its lower bounds, we get a negative answer to Problem \ref{prob:big}. Thus, Problems \ref{prob:big} and \ref{prob:marks} cannot both have a positive answer. Furthermore, if Problem \ref{prob:marks} has a positive answer, then by taking repeated free products of $\Gamma_k$ and $\Delta_k$ we find that the distance between Borel chromatic numbers for the shift graphs of marked groups with isomorphic Cayley graphs can be arbitrarily large.

We mention in passing an application of the groups $\Delta_3$ and $\Delta_3'$ appearing in Section \ref{sec:big} to another problem from \cite{M15}. Let $\CC$ be the class of all marked groups $\Gamma$ such that $\chi_B(F(2^\Gamma)) = d+1$, where $d$ is the size of $\Gamma$'s generating set. By Proposition 4.6 in \cite{KST}, the lower bound in Theorem \ref{th:marks} is sharp when $\Gamma$ and $\Delta$ are in $\CC$. In \cite{M15}, the claim is made that this is the only known case where the bound is sharp

We define a nonempty class of marked groups, $\D$, disjoint from $\CC$, such that the lower bound of Theorem \ref{th:marks} is sharp whenever $\Gamma \in \CC$ and $\Delta \in \D$.

Let $G$ be a graph on a set $X$. A subset $A \subset X$ is called \textit{connected} if the induced graph on it is connected. It is called \textit{biconnected} if $|A| \geq 2$ and $A \setminus \{x\}$ is connected for any $x \in A$. $G$ is called biconnected if $X$ is biconnected. $G$ is called a \textit{Gallai Tree} if for every maximal biconnected $A \subset X$, the induced graph on $A$ is a complete graph or an odd cycle. We have the following theorem from \cite{CMT} regarding Borel chromatic numbers and Gallai Trees:

\begin{theorem}
Let $G$ be a Borel graph on a standard Borel space of bounded degree $d$. If every connected component of $G$ is not a Gallai tree, $\chi_B(G) \leq d$.
\end{theorem}

Now, define $\D$ to be the set of marked groups $\Delta$ such that the Cayley graph of $\Delta$ is biconnected and $\chi_B(F(2^\Delta)) = d$, where $d$ is the size of $\Delta$'s generating set. The marked groups $\Delta_3$ and $\Delta_3'$ from Section \ref{sec:big} are in $\D$, as are the marked groups $(\Z,\{\pm 1, \pm 2\})$ (see page 11 in \cite{KST}) and $(\Z/n\Z,\{\pm 1\})$ for $n$ even. We now prove our main claim about $\D$:

\begin{prop}
Let $\Gamma \in \CC$ and $\Delta \in D$. Then the lower bound of Theorem \ref{th:marks} is sharp for the shift graph of $\Gamma * \Delta$.
\end{prop}
\begin{proof}
Suppose $\Gamma \in \CC$ and $\Delta \in \D$. Let $e$ be the size of $\Gamma$'s generating set and $d$ the size of $\Delta$'s. Then the lower bound of Theorem \ref{th:marks} for $\chi_B(F(2^{\Gamma * \Delta}))$ is $e+d$, and $e+d$ is also the degree of each vertex of $F(2^{\Gamma * \Delta})$. Therefore it suffices to prove the Cayley graph of $\Gamma * \Delta$ is not a Gallai tree.

Let $E \subset \Gamma * \Delta$ be an orbit of $\Delta$ under the left multiplication action of $\Delta$ on $\Gamma * \Delta$. The induced graph on $E$ is isomorphic to the Cayley graph of $\Delta$, so $E$ is biconnected. We claim $E$ is maximal among biconnected sets. Suppose $A \subset \Gamma * \Delta$ is connected and strictly contains $E$. Then $A$ contains some element $gy$ for $y \in E$ and $g$ a generator of $\Gamma$. However, every path from $E$ to $gy$ in the Cayley graph must pass through $y$, so $A$ is not biconnected.

Now, the Cayley graph of $\Delta$ cannot be an odd cycle or complete graph, since then the chromatic number of the graph would be $d+1$. Therefore we are done.
\end{proof}

We now return to the main topic of this subsection. Recall from Section \ref{sec:intro} that the analogue of Problem \ref{prob:main} where we consider edge chromatic numbers has an easier negative answer. However, using $\chi_B'$ to denote the Borel edge chromatic number, the following analogue of Problem \ref{prob:big} seems to be open:
\begin{prob}\label{prob:bigedge}
Let $\Gamma$ and $\Delta$ be two marked groups with isomorphic Cayley graphs. Must we have $$\left| \chi_B'(F(2^\Gamma)) - \chi_B'(F(2^\Delta)) \right| \leq 1 ?$$
\end{prob}
Taking inspiration from the example from Section \ref{sec:intro}, we might hope to use the groups $(\Z/2\Z)^{*2n}$ and $\Z^{*n}$ with their natural generators, which have isomorphic Cayley graphs. By labeling each edge according to the generator responsible for it, we can see that
\begin{equation}
\chi_B'(F(2^{(\Z/2\Z)^{*2n}})) = 2n.
\end{equation}
For $\Z^{*n}$, when $n>1$ we have only the following bounds (see Problem 5.37 in \cite{KM}):
\begin{equation}
2n \leq \chi_B'(F(2^{\Z^{*n}})) \leq 3n.
\end{equation}
Investigating edge colorings of this shift graph may be the best approach to resolving Problem \ref{prob:bigedge}. 

\subsection{Restrictions on the Marked Groups}

It is also natural to ask whether Theorem \ref{th:main} will continue to hold if we put certain restrictions on what the marked groups involved can be. For each $k \geq 3$, we pose the following problems arising in this way: We call a graph \textit{bipartite} if its chromatic number is less than or equal to 2.

\begin{prob}\label{prob:bipartite}
Are there marked groups $\Gamma_k,\Delta_k$ witnessing Theorem \ref{th:main} for this value of $k$ such that the Cayley graphs of the groups are bipartite?
\end{prob}

\begin{prob}\label{prob:tf}
Are there marked groups $\Gamma_k,\Delta_k$ witnessing Theorem \ref{th:main} for this value of $k$ such that each group is torsion free?
\end{prob}

In relation to Problem \ref{prob:bipartite}, we note that Proposition \ref{prop:2} remains true when the Borel chromatic number is replaced with the Baire measurable or measure chromatic number:

\begin{prop}
Let $(\Gamma,S)$ be an infinite marked group. Then $\chi_{BM}(F(2^\Gamma)) > 2$ and $\chi_M(F(2^\Gamma)) > 2$.
\end{prop}
\begin{proof}
For the Baire measurable number, the proof of Proposition \ref{prop:2} does not need to be modified, since Theorem 8.46 in \cite{K95} applies to all Baire measurable sets, not just Borel sets.

For the measure chromatic number the argument is very similar: Let $\mu$ be any of the measures $\mu_p$ from the proof of Lemma \ref{lem:measure}. Assume we have a $\mu$-measurable $2$-coloring $c$. In the notation of the proof of Proposition \ref{prop:2}, the action of $\ker \varphi$ on $F(2^\Gamma)$ is $\mu$-ergodic (since $\ker \varphi$ is infinite) so $\mu(A_i)$ must be 0 or 1 for each $i$. Again, since $A_1$ and $A_2$ partition $F(2^\Gamma)$ and $s \cdot A_1 = A_2$ for any $s \in S$, this is a contradiction.
\end{proof}

Combining this with some bounds from \cite{CM} gives us the following:
\begin{prop}
Let $(\Gamma,S)$ be an infinite marked group with a bipartite Cayley graph. Then $\chi_{BM}(F(2^\Gamma)) = 3$. If the connected component equivalence relation on this shift graph is hyperfinite, then we additionally have $\chi_{M}(F(2^\Gamma)) = 3$.
\end{prop}

Therefore, the Baire measurable analogue of Problem 6 has a negative answer, as does the measurable analogue when our connected component equivalence relations are hyperfinite.

For Problem \ref{prob:tf}, we conjecture that the answer is `yes' for each $k$. We offer the following pair of marked groups as a candidate to resolve the $k=3$ case:

Let $\Gamma = \Z \times \Z$ and $\Delta = \Z \rtimes_\varphi \Z$, where $\varphi$ is the homomorphism sending 1 to the inversion map and for each group we use the generators represented by $(\pm 2,0),(\pm 3, 0), (0,1)$. The Cayley graphs of these marked groups are clearly isomorphic, and their chromatic numbers are 3. However,
\begin{theorem}
$\chi_B(F(2^\Delta)) > 3$
\end{theorem}
\begin{proof}
Suppose we have a 3-coloring $c:F(2^\Delta) \rightarrow \{1,2,3\}$. Let $Z \leq \Delta$ be the subgroup generated by $(1,0)$. As in the proof of Theorem \ref{th:ab2}, for each $Z$-orbit $E$ the function
\begin{equation}
\begin{split}
i \mapsto c((n,0) \cdot x) \textnormal{ such that } x \in E \textnormal{ with } c(x) = i \ \ \ \ \ \ \ \\ \textnormal{ and } n \textnormal{ is the least positive integer with } c((n,0) \cdot x) \neq c(x) 
\end{split}
\end{equation}
is well defined and defines a 3-cycle in $S_3$, call it $\tau_E$. One can easily check that we must have $\tau_{(0,1) \cdot E} = \tau_E\inv$. Therefore, by keeping track of $\tau_E$, we get as in the proof of Theorem \ref{th:4} a 2-coloring of the $Z$-orbits of $F(2^\Delta)$. As in the proof of Theorem \ref{th:4}, if $c$ is Borel this contradicts $\chi_B(F(2^\Z)) = 3$.
\end{proof}

For $\Gamma$, if we make the same definitions we have $\tau_E = \tau_{(0,1) \cdot E}$. Therefore there is no obstruction as above, so we ask
\begin{prob}
What is $\chi_B(F(2^\Gamma))$?
\end{prob}
We conjecture that it is 3. Note that $\Gamma$ contains $\Z^2$ with its usual generators as a marked subgroup, and so this conjecture implies that $\chi_B(F(2^{\Z^2})) = 3$. Recently, it was announced in \cite{GJKS} that this is indeed the case, removing another potential obstruction.

\section*{Acknowledgements}
This research was supported by NSF Grant DMS-1464475.

We thank A.S. Kechris for his guidance throughout the conduction of this research, his many helpful comments on drafts of this paper, and his idea to include Section \ref{sec:iso}.



\pagebreak

\end{document}